\documentclass[11pt]{amsart}

\usepackage{lscape}

\newcounter{cnt1}
\newcounter{cnt2}
\newcommand{\blr}{\begin{list}{$(\roman{cnt1})$}
    {\usecounter{cnt1} \setlength{\topsep}{0pt}
        \setlength{\itemsep}{0pt}}}
\newcommand{\bla}{\begin{list}{$($\alph{cnt2}$)$}
    {\usecounter{cnt2} \setlength{\topsep}{0pt}
        \setlength{\itemsep}{0pt}}}
\newcommand{\el}{\end{list}}
\newtheorem{thm}{Theorem}

\newtheorem{cor}[thm]{Corollary}
\newtheorem{ex}[thm]{Example}

{\newtheorem{Def}[thm]{Definition}
\newtheorem{prop}[thm]{Proposition}
\newtheorem{rem}[thm]{Remark}
\newcommand{\Rem}{\begin{rem} \rm}
\newcommand{\bdfn}{\begin{Def} \rm}
\newcommand{\edfn}{\end{Def}}

\begin{document}
\large
\title[ Subdifferentiability]{Subdifferentiability and polyhedrality of the norm}
\author[Rao]{T. S. S. R. K. Rao}
\address[T. S. S. R. K. Rao]
{Department of Mathematics\\
Shiv Nadar University \\
Delhi (NCR) \\ India,
\textit{E-mail~:}
\textit{srin@fulbrightmail.org}}

\begin{abstract}
Let $X$ be an infinite dimensional real Banach space. In this paper, motivated by the work of Contreras and Pay$\acute{a}$ (\cite{CP}) we study Banach space properties that are necessary or sufficient to ensure subdifferentiability of the norm at all unit vectors.\end{abstract}
  \maketitle
\section { Introduction}
Let $X$ be a Banach space. We recall from \cite{FP} that the norm is said to be strongly subdifferentiable (SSD) at a unit vector $x$, if $lim_{t \rightarrow 0^+}\frac{\|x+ty\|-\|x\|}{t}$ exists uniformly for $y$ in the unit ball $X_1$ of $X$. Let $S_X$ denote the unit sphere. Clearly if $x$ is a SSD point, so is $-x$. Also when $X$ is canonically embedded in its bidual $X^{\ast\ast}$, $x$ is a SSD point in $X$ if and only if it is a SSD point of $X^{\ast\ast}$.
\vskip 1em
It is easy to see that the norm  on any finite dimensional Banach space is strongly subdifferentiable. It is known that in a Banach algebra, the norm is SSD at the  identity (see \cite{FP}) but like in $C([0,1])$, this (and its negative) can be the only point of SSD for the norm (see \cite{CP}, Theorem 2).
\vskip 1em
We recall from that $X$ is said to be weakly Hahn-Banach smooth, if every norm-attaining functional in $S_{X^\ast}$ is a point of weak$^\ast$-weak continuity for the identity map on $S_{X^\ast}$. See Page 147, Lemma III.2.14 and several examples from Chapter III of \cite{HWW} where this condition is satisfied. Theorem 2 in \cite{CP} shows for $C^\ast$-algebras, the equivalence of weakly Hahn-Banach smoothness and strong subdifferentiability of the norm (they assume the stronger hypothesis of coincidence of weak$^\ast$ and weak topologies on $S_{X^\ast}$, but it is clear from the proof that the coincidence at norm attaining functionals in $S_{X^\ast}$ is enough). We are interested in exploring similar results in the commutative set-up, including in concrete function algebras , Banach spaces $X$ whose dual is isometric to $L^1(\mu)$, for a positive measure $\mu$, called $L^1$-predual spaces. See \cite{L}.
\vskip 1em
We say that a Banach space has the weak Namioka-Phelps property, if every norm attaining functional in $S_{X^\ast}$ is a point of weak$^\ast$-norm continuity for the identity map on $S_{X^\ast}$. See \cite{R2} . It is known that for any Hilbert space $H$, the space of compact operators ${\mathcal K}(H)$ has this property, see \cite{LM}. In the commutative case, it is easy to see that for any discrete set $\Gamma$, $c_0(\Gamma)$ has this property.
See also Chapter VI of \cite{HWW} for more examples among non-reflexive Banach spaces. 
\vskip 1em

We first show that for $C^\ast$-algebras, weak Hahn-Banach smoothness is equivalent to weak Namioka-Phelps property.
We also show that for any Banach space $X$ with this property, norm is SSD on $S_X$.
\vskip 1em
These notions play an important role in various branches of Analysis. To give an application to approximation theory, we recall (an equivalent formulation) from \cite{GI} that a proximinal subspace $Y \subset X$ (i.e., $P(x)=\{y \in Y: d(x,Y)=\|x-y\|\}$ is non-empty for all $x \in X$) is said to be strongly proximinal, if for every $x \in X$, every sequence $\{y_n\}_{ n \geq 1} \subset Y$ such that $d(x,Y)= lim\|x-y_n\|$, there is a subsequence $\{z_{n_k}\} \subset P(x)$ such that $\|y_{n_k}-z_{n_k}\|\rightarrow 0$.
\vskip 1em
It was shown in \cite{GI} that  $x^\ast \in S_{X^\ast}$ is a SSD point if and only if $ker(x^\ast)$ is a strongly proximinal subspace of $X$. An important problem is to decide, when is every weak$^\ast$-closed finite codimensional subspace is strongly proximinal in $X^\ast$?
\vskip 1em
Thus in view of our remarks above, subdifferentiability of the norm on $S_{X}$ is equivalent to every weak$^\ast$-closed (hence proximinal) hyperplane in $X^\ast$ is strongly proximinal.
\vskip 1em
For any $x \in S_X$, the state space $S_x = \{x^\ast \in S_{X^\ast}: x^\ast(x)=1\}$ is a weak$^\ast$-closed exposed (by $x$) face of $X^\ast_1$. It is well known that (see \cite{DGZ}) the subdifferential limit $lim_{ t \rightarrow 0^+}\frac{\|x+ty\|-\|x\|}{t}$ exists and is $sup\{x^\ast(y): x^\ast \in S_X\}$. $x$ is a SSD point if and only if $x$ strongly exposes $S_x$, in the sense, for any sequence $\{x^\ast_n\}_{ n \geq 1} \subset X^\ast_1$, $x^\ast_n(x) \rightarrow 1$ implies $d(x^\ast_n,S_x) \rightarrow 0$. See \cite{FP}. See also \cite{R3} for applications of these ideas in spaces of operators.
\vskip 1em
If $\rho: S_X \rightarrow 2^{X^\ast}$ defined by $\rho(x)=S_x$ is the duality map, then  an equivalent formulation of $x$ being a SSD point is that, it is a point of norm-norm upper-semi-continuity for $\rho$. This means, given $\epsilon >0$, there is a $\delta>0$ such that for $z\in S_X$ with $\|x-z\|< \delta$, $S_z \subset S_x+ \epsilon X^\ast_1$. See \cite{FP}.
\vskip 1em
The results presented here are valid in either scalar field. For convenience we most often work with real scalars.
\section{Main Results}
We note that the properties of norm being SSD and the weak Namioka-Phelps properties are hereditary. By an application of the Bishop-Phelps theorem, if a separable Banach space has the weak Namioka-Phelps property, then $X^\ast$ is a separable Banach space. Thus if $X$ has the weak Namioka-Phelps property, every separable subspace has separable dual. Therefore $X^\ast$ has the Radon-Nikodym property (see \cite{DU})
. These deductions allow us to deal with sequences in the following theorem.
\begin{thm}
Let $X$ be a Banach space satisfying weak Namioka-Phelps property. The norm on $X$ is SSD.	
\end{thm}
\begin{proof}
	Let $x \in S_x$. We need to show that $x$ strongly exposes $S_x$. Let $\{x^\ast_n\}_{n \geq 1} \subset X^\ast_1$ and $x^\ast_n(x) \rightarrow 1$. Clearly all weak$^\ast$ accumulation points of this sequence are in $S_x$. Since every point of $S_x$ is a point of weak$^\ast$-norm continuity, we get that $d(x^\ast_n,S_x) \rightarrow 0$.
\end{proof}
Let $Y \subset X$ be a proximinal subspace. Consider the quotient space $X/Y$. If a unit vector $x^\ast \in Y^\bot$ attains its norm, $d(x,Y)=1 =x^\ast(x)$. Since $Y$ is proximinal, let $1 = d(x,Y)=\|x-y\|=x^\ast(x-y)$ so that $x^\ast$ attains its norm on $X$. Now it is easy  to see that if $X$ has the weak Hahn-Banach smooth or Namioka-Phelps property, so does the quotient space $X/Y$.
\begin{rem}
	We recall from Chapter 1 of \cite{HWW} that a closed subspace $Y\subset X$ is a $M$-ideal, if there is a linear projection $P: X^\ast \rightarrow X^\ast$ such that $ker(P) = Y^\bot$ and $\|x^\ast\|=\|P(x^\ast)\|+\|x^\ast - P(x^\ast)\|$ for all $x^\ast \in X^\ast$. Such a subspace is a proximinal subspace. Consequently if $X$ has the weakly Hahn-Banach smooth or Namioka-Phelps property, so does $X/Y$.
\end{rem}
\vskip 1em
$M$-ideals in $C^\ast$-algebras are precisely closed two-sided ideals and in $L^1$-predual spaces, the quotient by a $M$-ideal is again a $L^1$-predual space.
\vskip 1em

We next show that the converse of the above theorem is valid in $C^\ast$-algebras. 
\vskip 1em
For a family $\{X_{\alpha}\}_{\alpha \in \Delta}$ of Banach spaces indexed by a set $\Delta$, we denote the $c_0$-sum by $X = \bigoplus_{c_0} X_{\alpha}$. In what follows we omit writing the index set.
\begin{prop} Suppose each $X_{\alpha}$ has the weak Namioka-Phelps property. Then $X$ has the weak Namioka-Phelps property.
	
\end{prop}
\begin{proof}
	Let $\Lambda \in \bigoplus_1 X^\ast_{\alpha}=X^\ast$ be a unit vector attaining its norm. It is easy to see that $\Lambda$ has only finitely many non-zero coordinates and at each non-zero coordinate $\alpha$, the corresponding functional in the coordinate space $X^\ast_{\alpha}$ attains its norm.
	Thus we assume without loss of generality that $X = M \bigoplus_{\infty} N$ and both $M,~N$ have the weak Namioka-Phelps property.
	\vskip 1em
	Suppose $(m^\ast,n^\ast) \in S_X$ is norm attaining. $\|m^\ast\|+\|n^\ast\| =1$, suppose both components are non-zero and $m^\ast(m)+n^\ast(n)=1=max\{\|m\|,\|n\|\}$.
	$\frac{m^\ast}{\|m^\ast\|},~\frac{n^\ast}{\|n^\ast\|}$ attain norm at unit vectors $m~,n$ respectively.
	Let  nets $m^\ast_{\gamma} \rightarrow m^\ast$, $n^\ast_{\gamma} \rightarrow n^\ast$ in the weak$^\ast$-topology, $lim\{\|m^\ast_{\gamma}\|+\|n^\ast_{\gamma}\|\}=1$. Since $lim\{m^\ast_{\gamma}(m)+n^\ast_{\gamma}(n)\}=1$ we get that $\frac{m^\ast_{\gamma}}{\|m^\ast_{\gamma}\|} \rightarrow \frac{m^\ast}{\|m^\ast\|}$ in the weak$^\ast$ topology of $M^\ast$ and a similar statement is true of $\frac{n^\ast}{\|n^\ast\|}$. Thus by assumption on the component spaces, we get the conclusion. Similar arguments work, when one of the components is zero.
\end{proof}
\begin{thm}
	A $C^\ast$ algebra $A$ has the weak Hahn-Banach smooth property if and only if it has the weak Namioka-Phelps property.
\end{thm}
\begin{proof}
	Suppose $A$ has the weak Hahn-Banach smooth property. By our remarks in the Introduction and by Theorem 2 in \cite{CP} we have that $A$ is isometric to a $c_0$ sum of a family of spaces of the type ${\mathcal K}(H)$ for a family of Hilbert spaces. Since each component space has the Namioka-Phelps property, the conclusion follows from the above proposition.
\end{proof}
In the next example we illustrate the commutative case.
\begin{ex}
	Let $\Omega$ be a locally compact Hausdorff space. Suppose the supremum norm on $C_0(\Omega)$ is SSD at all unit vectors. Then $\Omega$ is a discrete set.
	\vskip 1em
	To see this first assume every singleton in $\Omega$ is a $G_{\delta}$. For $\omega \in \Omega$, let $f \in C_0(\Omega)$, $0 \leq f \leq 1$, $f^{-1}(1)=\{\omega\}$. Thus $S_f = \{\delta({\omega})\}$. Since $f$ is a SSD point, $f$ strongly exposes $\delta(\omega)$. Thus $\{\omega\}$ is an isolated point. Hence $\Omega$ is a discrete set. In the general case, as SSD of norm is a hereditary property, we apply the arguments to a separable self-adjoint sub-algebra to get the required conclusion.
\end{ex}
\begin{thm}
	Let $X$ be an infinite dimensional $L^1$-predual space. If the norm is SSD then $X^\ast$ is isometric to $\ell^1(\Gamma)$ for an infinite discrete set $\Gamma$.
\end{thm}
\begin{proof}
	Let $Y \subset X$ be a separable subspace. By Lemma 6 on page 227 in \cite{LM} we get that there is a separable $L^1$-predual space $Z$ such that $Y \subset Z \subset X$. If $Z^\ast$ is not separable, by Theorem 4 on page 226 of \cite{L} $Z$ and hence $X$ contains an isometric copy of continuous functions on the Cantor set. Since norm SSD is a hereditary property, in view of Example 4, we get a contradiction. Therefore $Y^\ast$ is a separable space. As before we conclude that $X^\ast$ has the Radon-Nikodym property. Since $X^\ast = L^1(\mu)$, this implies $X^\ast$ is isometric to $\ell^1(\Gamma)$ for some discrete set $\Gamma$.
\end{proof}
We recall from \cite{GI} that the norm is said to be quasi-polyhedral, if for $x \in S_X$, there is a $\delta >0$ such that $y \in S_X$, $\|y-x\|< \delta$ implies $S_y \in S_X$. Here again we note that it is enough to show that extreme points of the weak$^\ast$-closed face,  $S_y$ are in $S_x$. It is easy to see that in this case the norm is SSD. On the other hand since in any finite dimensional space the norm is SSD, the converse need not be true.
\vskip 1em
\begin{ex}
	Let $\Omega$ be a compact set and let $E \subset \Omega$ be a clopen set. If $\|f\|=1$ and $\|f -\chi_{E}\|<1$, then if $f(\omega)=1$ then $\omega \in E$. Thus $S_f \subset S_{\chi_E}$. Hence $\chi_E$ is a quasi-polyhedral point.
\end{ex}
The next theorem shows that if the norm on a $L^1$-predual space is SSD, then it is quasi-polyhedral.
\begin{thm}
Let $X$ be a $L^1$-predual space. If the norm is SSD then it is quasi polyhedral. Hence it is a quasi polyhedral predual of $\ell^1(\Gamma)$ for some discrete set $\Gamma$.	
\end{thm}
\begin{proof}
	Let $x \in S_X$ . By hypothesis it follows that $x$ is a SSD point of $X^{\ast\ast}$. Since $X^\ast=L^1(\mu)$, we have $X^{\ast\ast} = C(K)$, for some Compact Hyperstonean space $K$ (see \cite{L} page 95) . We have from Theorem 1 of \cite{CP} that the set of extreme points of $S_x$ (now considered in $C(K)$) is a clopen set. Therefore by Example 7, we get that $x$ is a quasi-polyhedral point for the norm on $C(K)$. Hence $x$ is a quasi-polyhedral point for the norm on $X$. Therefore $X$ is a quasi-polyhedral space.
	
\end{proof}
The following corollary is now easy to prove.
\begin{cor}
	Let $K$ be a compact Hausdorff space. The norm $C(K)$ is polyhedral at a dense set of points in $S_{C(K)}$ if and only if $K$ is a totally disconnected space (i.e., $K$ has a base consisting of clopen sets). See Theorem 3 of \cite{CP}.
\end{cor}
We recall that a function algebra (over complex scalars) $A$ is a closed, point separating  subalgebra, containing constants of $C(\Omega)$ for a compact set $\Omega$. Next theorem is another commutative version of Theorem 2 in \cite{CP}.
\begin{thm}
	Let $A$ be a function algebra which is weakly Hahn-Banach smooth. Then $A$ is isometric to $\ell^{\infty}(n)$ for some positive integer $n$.
\end{thm}
\begin{proof}
	Since $A$ is weakly Hahn-Banach smooth, any norm attaining unit vector of $A^\ast$, by Mazur's theorem, in conjugation with the Krein-Milman theorem,  is in the norm closed convex hull of extreme points of $A^\ast_1$. Thus  by the Bishop-Phelps theorem, $A^\ast_1$ is the norm-closed convex hull of its its extreme points. Since extreme points come from weak-peak points, we see that for any extreme point $x^\ast \in A^\ast_1$,  $A^\ast = span\{x^\ast\}\bigoplus_1 N$ for some closed subspace $N$ (see \cite{HWW} page 5). If $T$ denotes the unit circle and $y^\ast$ is an extreme point not in $T\{x^\ast\}$, then $y^\ast \in N$. In particular $\|\alpha x^\ast+\beta y^\ast\|=|\alpha|+|\beta|$, for any scalars $\alpha,~\beta$.
	\vskip 1em
	It is easy to see that there is a set of extreme points $\Gamma$ such that no two vectors are dependent and $T\Gamma$ consists of all extreme points of the unit ball. It is now easy to see that $A^\ast$ is isometric to $\ell^1(\Gamma)$. $A$ is in particular a $L^1$-predual space. It now follows from Corollary 4.2 in \cite{ERR} that $A$ is self adjoint. Thus by Example 4 again we get that $A$ is isometric to $\ell^{\infty}(n)$.
\end{proof}


\begin{thebibliography}{99}

\bibitem{DGZ} R. Deville, G. Godefroy and V. Zizler, {\em Smoothness and renormings in Banach spaces}, Pitman Monographs and Surveys in Pure and Applied Mathematics, vol. 64, Longman Scientific and Technical, Harlow, 1993.
\bibitem{CP} M. D. Contreras, R.  Pay$\acute{a}$ and W.  Werner {\em $C^\ast$-algebras that are I-rings}, J. Math. Anal. Appl. 198 (1996) 227--236.
\bibitem{DU}  J. Diestel and J. J. Uhl, {\em Vector measures}, Mathematical Surveys 15, American Mathematical Society, Providence, RI, 1977.
\bibitem{ERR} A. J. Ellis, T. S. S. R. K. Rao, A. K. Roy and U. Uttersurd, {\em Facial characterizations of complex Lindenstrauss spaces}, Trans. Amer. Math. Soc., 268 (1981) 173--186.
\bibitem{FP} C. Franchetti and R.  Pay$\acute{a}$, {\em Banach spaces with strongly subdifferentiable norm}, Boll. Un. Mat. Ital. B  7 (1993) 45--70.
 \bibitem{HWW} P. Harmand, D. Werner and W. Werner, {\em $M$-ideals in Banach spaces and Banach algebras}, Springer LNM 1547, Berlin 1993.
   \bibitem{GI} G. Godefroy and V. Indumathi, {\em Strong proximinality and polyhedral spaces}, Rev. Mat. Complut. 14 (2001)
105--125.
  \bibitem{L} H. E.  Lacey, {\em  The isometric theory of classical Banach spaces}, In: DieGrundlehren der Mathematischen Wissenschaften, Band, vol. 208, pp. x+270. Springer, New York (1974).
  \bibitem{LM} A. T. M. Lau and P. F. Mah, {\em Quasi normal structures for certain spaces of operators on a Hilbert space}, Pacific J. Math. 121 (1986) 109--118.
  
  \bibitem{R3} T. S. S. R. K. Rao, {\em Subdifferential set of an operator}, Monatsh. Math. 199 (2022) 891--898.
 \bibitem{R2} T. S. S. R. K. Rao, {\em Spaces with the Namioka-Phelps property have trivial $L$-structure}, Archiv der Mathematik, 62 (1994) 65--68.
  \end{thebibliography}
\end{document}